\documentclass[12pt]{amsart}

\usepackage{graphicx}

\newtheorem{theorem}{Theorem}[section]
\newtheorem{proposition}[theorem]{Proposition}
\newtheorem{lemma}[theorem]{Lemma}
\newtheorem{corollary}[theorem]{Corollary}

\theoremstyle{definition}

\newtheorem{example}[theorem]{Example}
\newtheorem{problem}[theorem]{Problem}
\newtheorem{question}[theorem]{Question}

\newtheorem{remark}[theorem]{Remark}

\newcommand{\ZZ}{ \ensuremath{\mathbb{Z}}}

\newcommand{\RR}{ \ensuremath{\mathbb{R}}}

\newcommand{\Tor}{\ensuremath{\mathrm{Tor}}\hspace{1pt}}

\def\cocoa{{\hbox{\rm C\kern-.13em o\kern-.07em C\kern-.13em o\kern-.15em A}}}

\newcommand{\mult}{\delta}
\newcommand{\type}{\mathrm{type}}
\newcommand{\cost}{\mathrm{cost}}
\newcommand{\aaa}{\mathbf{a}}

\newcommand{\cano}{\omega}
\newcommand{\supp}{\mathrm{supp}}
\newcommand{\ee}{\mathbf{e}}
\newcommand{\lk}{\mathrm{lk}}

\newcommand{\coker}{\ensuremath{\mathrm{coker}}\hspace{1pt}}

\begin{document}

\title[Uniformly Cohen-Macaulay simplicial complexes]{Uniformly Cohen--Macaulay simplicial complexes
and almost Gorenstein* simplicial complexes}

\author{Naoyuki Matsuoka}
\address{
Naoyuki Matsuoka,
Department of Mathematics,
School of Science and Technology,
Meiji University,
1-1-1 Higashi-mita, Tama-ku, Kawasaki-shi, Kanagawa 214-8571, Japan.
}
\email{naomatsu@meiji.ac.jp}

\author{Satoshi Murai}
\address{
Satoshi Murai,
Department of Pure and Applied Mathematics,
Graduate School of Information Science and Technology,
Osaka University, Suita, Osaka 565-0871, Japan.
}
\email{s-murai@ist.osaka-u.ac.jp}

\thanks{
The first author was partially supported by JSPS KAKENHI 26400054.
The second author was partially supported by JSPS KAKENHI 25400043.}


\begin{abstract}
In this paper, we study simplicial complexes whose Stanley--Reisner rings are almost Gorenstein and have $a$-invariant zero.
We call such a simplicial complex an almost Gorenstein* simplicial complex.
To study the almost Gorenstein* property, we introduce a new class of simplicial complexes which we call uniformly Cohen--Macaulay simplicial complexes.
A $d$-dimensional simplicial complex $\Delta$ is said to be uniformly Cohen--Macaulay
if it is Cohen--Macaulay and, for any facet $F$ of $\Delta$, the simplicial complex $\Delta \setminus\{F\}$ is Cohen--Macaulay of dimension $d$.
We investigate fundamental algebraic, combinatorial and topological properties of these simplicial complexes,
and show that almost Gorenstein* simplicial complexes must be uniformly Cohen--Macaulay. By using this fact, we show that every almost Gorenstein* simplicial complex can be decomposed into those of having one dimensional top homology. Also, we give a combinatorial criterion of the almost Gorenstein* property for simplicial complexes of dimension $\leq 2$.
\end{abstract}

\maketitle

\section{Introduction}
In this paper, we study the almost Gorenstein property of Stanley--Reisner rings.
Let $K$ be a field, $R=\bigoplus_{k \in \ZZ} R_k$ a graded $K$-algebra of Krull dimension $\geq 1$, and let $\omega_R$ be its canonical module.
The algebra $R$ is said to be {\em almost Gorenstein} if 
$R$ is Cohen--Macaulay and there is a short exact sequence of degree $0$
\begin{align}
\label{1-1}
0 \longrightarrow R \stackrel{\phi}{\longrightarrow} \cano_R (-a) \longrightarrow C=\coker \phi \longrightarrow 0
\end{align}
such that the multiplicity of $C$ is equal to the number of minimal generators of $C$,
where $a=-\min\{k: (\cano_R)_k \ne 0\}$ is the {\em $a$-invariant} of $R$
and where $M(k)$ denotes the graded module $M$ with grading shifted by degree $k$.
Note that if $R$ is Gorenstein, then it is almost Gorenstein and $C$ is the zero module.

The almost Gorenstein property was first introduced by Barucci and Fr\"oberg \cite{BF} for $1$-dimensional local rings which are analytically unramified.
Then the definition was extended to all $1$-dimensional local rings by Goto--Matsuoka--Phuong \cite{GMP} and to all local and graded rings by Goto--Takahashi--Taniguchi \cite{GTT}. 
Since the almost Gorenstein property for graded rings was introduced very recently,
not many are known on this property and it is expected to understand what kinds of rings are almost Gorenstein and to find various examples of almost Gorenstein rings.
In this paper, we address this problem for Stanley--Reisner rings.

We first recall basics on simplicial complexes and Stanley--Reisner rings.
Let $\Delta$ be a simplicial complex on $[n]=\{1,2,\dots,n\}$.
Thus $\Delta$ is a collection of subsets of $[n]$ such that,
for any $F \in \Delta$ and $G \subset F$, one has $G \in \Delta$.
Elements of $\Delta$ are called {\em faces} of $\Delta$
and the maximal faces (under inclusion) are called {\em facets} of $\Delta$.
The {\em dimension} of $\Delta$ is the integer
$\dim \Delta=\max \{\# F: F \in \Delta\} -1$,
where $\#X$ denotes the cardinality of a finite set $X$.
For every simplicial complex $\Delta$ on $[n]$,
we write $|\Delta|$ for its natural geometric realization which is the union of
convex hulls of $\{\ee_i: i \in F\}$ for all $F \in \Delta$, where $\ee_i$ is the $i$th unit vector of $\RR^n$.
Next, we recall Stanley--Reisner rings.
Let $K$ be an infinite field and $S=K[x_1,\dots,x_n]$ the polynomial ring with each $\deg x_i=1$.
For a simplicial complex $\Delta$ on $[n]$,
its {\em Stanley--Reisner ring} (over $K$) is the quotient ring
$$K[\Delta]=S/(x_F: F \subset [n],\ F \not \in \Delta),$$
where $x_F= \prod _{i \in F} x_i$.
We say that a simplicial complex $\Delta$ is {\em Cohen--Macaulay} (CM for short)
over $K$ if $K[\Delta]$ is a Cohen--Macaulay ring \cite{BH}.
The CM property depends on the base field $K$,
but we will omit information on $K$ when it is clear.

To study the almost Gorenstein property, it is convenient to know which graded $K$-algebra $R$ admits the short exact sequence \eqref{1-1}, that is, when there is an injection $R \to \omega_R(-a)$. If $R$ is a domain, then it is clear that such an injection exists. However, if $R$ is not a domain, then the existence of such an injection is non-trivial. The first main result of this paper is a simple combinatorial criterion of simplicial complexes whose Stanley--Reisner rings admit such an injection with $a=0$.

For a simplicial complex $\Delta$ on $[n]$ and a face $F \in \Delta$,
the simplicial complex
$$\cost_\Delta(F)=\{G \in \Delta: G \not \supset F\}$$
is called the {\em contrastar} of $F$ in $\Delta$.
We say that a $d$-dimensional simplicial complex $\Delta$ is {\em uniformly Cohen--Macaulay}
if $\Delta$ is CM and, for any facet $F \in \Delta$,
$\cost_\Delta(F)=\Delta \setminus \{F\}$ is CM of dimension $d$.
(We note that if $\Delta$ is CM and $F$ is its facet, then $\cost_\Delta(F)$ has dimension $d$ unless $\Delta$ is a simplex.)
For topological spaces (or simplicial complexes) $X \supset Y$,
we write $H_k(X)$ and $H_k(X,Y)$ for the $k$th homology group
of $X$ and of the pair $(X,Y)$ with coefficients in $K$, respectively.

\begin{theorem}
\label{main1}
Let $\Delta$ be a CM simplicial complex on $[n]$ of dimension $d-1 \geq 1$.
The following conditions are equivalent.
\begin{itemize}
\item[(i)] $\Delta$ is uniformly CM.
\item[(ii)] There is an injective $S$-homomorphism $\phi: K[\Delta] \to \omega_{K[\Delta]}$ of degree $0$.
\item[(iii)] For any facet $F \in \Delta$, $\dim_K H_{d-1}(\Delta) > \dim_K H_{d-1}(\cost_\Delta(F))$.
\item[(iv)] For any point $p \in |\Delta|$, the map $\imath^*: H_{d-1}(|\Delta|) \to H_{d-1}(|\Delta|,|\Delta|-p)$ induced by the inclusion $\imath:|\Delta| -p \to |\Delta|$
is non-zero.
\end{itemize}
\end{theorem}

The uniformly CM property is a new property of simplicial complexes and is itself an interesting property.
Since the CM property is a topological property,
the fourth condition of the theorem shows that
the uniformly CM property is a topological property,
that is, only depends on the geometric realization and a base field $K$.
The second condition implies an interesting enumerative consequence.
It implies that if $(h_0,h_1,\dots,h_d)$ is the $h$-vector of the uniformly CM simplicial complex,
then we have $\sum_{i=0}^k h_{d-i} \geq \sum_{i=0}^k h_i$ for all $k$.
This is an analogue of Stanley's inequality for Ehrhart $\delta$-vectors \cite{St2}.
Also, the third condition gives a simple way to check the uniformly CM property.

The definition of the uniformly CM property is similar to that of
the $2$-CM property introduced by Baclawski \cite{Ba}.
A $d$-dimensional simplicial complex $\Delta$ is said to be $2$-CM if,
for every vertex $v$ of $\Delta$, $\cost_\Delta(v)$ is CM of dimension $d$.
In particular,
for $1$-dimensional simplicial complexes,
the 2-CM and uniformly CM property are the 2-vertex and 2-edge connectivity of graphs.
As we see later, it is not hard to see that $2$-CM simplicial complexes are uniformly CM (Proposition \ref{3.4}).
These fundamental properties of uniformly CM simplicial complexes
will be discussed in the first part of the paper.

In the latter half of the paper, we study the almost Gorenstein property of Stanley--Reisner rings.
We are interested in the following special class of almost Gorenstein Stanley--Reisner rings.
We say that a simplicial complex $\Delta$ is {\em almost Gorenstein*} (over $K$)
if $K[\Delta]$ is almost Gorenstein and $\widetilde H_{\dim \Delta}(\Delta) \ne 0$,
where $\widetilde H_k(\Delta)$ is the $k$th reduced homology group of $\Delta$ with coefficients in $K$.
These are natural generalizations of Gorenstein* simplicial complexes which are simplicial complexes $\Delta$ such that $K[\Delta]$ is Gorenstein and $\widetilde H_{\dim \Delta}(\Delta) \ne 0$.
Since this non-vanishing condition on the homology implies that the $a$-invariant of $K[\Delta]$ is zero by Hochster's formula \cite[II Theorem 4.1]{St},
Theorem \ref{main1} and the existence of a short exact sequence \eqref{1-1} say that almost Gorenstein* simplicial complexes are uniformly CM.

We say that an almost Gorenstein* simplicial complex $\Delta$ is {\em indecomposable} if $\widetilde H_{\dim \Delta}(\Delta)\cong K$.
Our main result is that every almost Gorenstein* simplicial complex can be decomposed into indecomposable ones in the following sense.
Let $\Delta,\Gamma$ and $\Sigma$ be $(d-1)$-dimensional simplicial complexes.
We say that $\Delta$ is the {\em ridge sum} of $\Gamma$ and $\Sigma$
if $\Delta=\Gamma \cup \Sigma$ and $\Gamma \cap \Sigma$ is a $(d-2)$-dimensional simplex.
As we prove later,
if $\Delta$ is the ridge sum of $\Gamma$ and $\Sigma$,
then $\Delta$ is almost Gorenstein* if and only if $\Gamma$ and $\Sigma$ are almost Gorenstein* (Lemma \ref{4.5}).

\begin{theorem}
\label{main2}
Let $\Delta$ be an almost Gorenstein* simplicial complex of dimension $d-1 \geq 1$.
If $\Delta$ is not indecomposable,
then there are $(d-1)$-dimensional almost Gorenstein* simplicial complexes $\Gamma$ and $\Sigma$ such that
$\Delta$ is a ridge sum of $\Gamma$ and $\Sigma$.
\end{theorem}

The above theorem says that every almost Gorenstein* simplicial complex is a ridge sum of (many) indecomposable ones.
Thus, to study the almost Gorenstein* simplicial complexes,
it is essential to consider indecomposable ones.
This fact gives some insight on almost Gorenstein* simplicial complexes of dimension $\leq 2$.
It is not hard to see that $1$-dimensional indecomposable almost Gorenstein* simplicial complexes are Gorenstein* (in particular, they are cycles).
By using this,
we classify almost Gorenstein Stanley--Reisner rings of Krull dimension $2$ (Proposition \ref{4.7}).
We also show that, in two dimensional case, the indecomposable almost Gorenstein* simplicial complexes are exactly the uniformly CM simplicial complexes $\Delta$ with $\widetilde H_2(\Delta)\cong K$ (Proposition \ref{4.8}).
This provides a combinatorial criterion of the $2$-dimensional almost Gorenstein* simplicial complexes by Theorems \ref{main1} and \ref{main2}.

\section{Uniformly Cohen--Macaulay simplicial complexes}

In this section, we study the uniformly CM property.
We first introduce necessary notations.
Throughout the paper,
$S=K[x_1,\dots,x_n]$ is the polynomial ring over an infinite field $K$.
Let $\ee_i \in \ZZ^n$ be the $i$th unit vector of $\ZZ^n$.
We consider the $\ZZ^n$-grading of $S$ defined by $\deg x_i = \ee_i$.
For a $\ZZ^n$-graded $S$-module $M$ and for $\aaa=(a_1,\dots,a_n) \in \ZZ^n$,
let $M_\aaa$ be the graded component of $M$ of degree $\aaa$
and let $\supp(\aaa)=\{i \in [n]: a_i \ne 0\}$.
For a subset $F \subset [n]$,
we write $\ee_F= \sum_{i \in F} \ee_i$.

The key idea to prove Theorem \ref{main1} is Gr\"abe's structure theorem for canonical modules of Stanley--Reisner rings.
We quickly recall Gr\"abe's result.
Fix a simplicial complex $\Delta$ of dimension $d-1 \geq 1$.
For faces $G \subset F \in \Delta$,
let
$$\imath^*: H_{d-1}(\Delta,\cost_\Delta(G)) \to H_{d-1}(\Delta,\cost_\Delta(F))$$
be the map 
appearing in the long exact sequence of the triple $(\Delta,\cost_\Delta(F),\cost _\Delta(G))$.
Gr\"abe \cite[Theorem 4]{Gr} proved the isomorphism
\begin{align}
\label{2-1}
\omega_{K[\Delta]} \cong \bigoplus_{\aaa \in \ZZ^n_{\geq 0},\ \supp(\aaa) \in \Delta} H_{d-1}\big(\Delta,\cost_\Delta\big(\supp(\aaa) \big) \big),
\end{align}
where the $\aaa$th graded component of the right-hand side is $H_{d-1}(\Delta,\cost_\Delta(\supp(\aaa)))$
and where we consider that $H_{d-1}(\Delta,\cost_\Delta(\supp(\aaa) ))= H_{d-1}(\Delta)$ if $\supp(\aaa)=\emptyset$,
and determined the multiplication structure of the right-hand side as follows.

\begin{lemma}[Gr\"abe]
\label{2.1}
With the same notation as above,
the multiplication of $x_l$ from the $\aaa$th component of the right-hand side of \eqref{2-1} to its $(\aaa+\ee_l)$th component is the following map
\begin{align*} 
\small
\begin{cases}
\mbox{$0$-map}, & \mbox{if }\supp(\aaa +\ee_l) \not \in \Delta,\\
\mbox{identity map}, & \mbox{if } l \in \supp(\aaa) \in \Delta,\\
\imath^* : H_{d-1}(\Delta,\cost_\Delta(\supp(\aaa ))) \to H_{d-1}(\Delta,\cost_\Delta(\supp(\aaa+\ee_l))),&
\mbox{otherwise}.
\end{cases}
\end{align*}
\end{lemma}

In particular,
Lemma \ref{2.1} says that the multiplication
$\times x_F : (\cano_{K[\Delta]})_0 \to (\cano_{K[\Delta]})_{\ee_F}$
can be identified with the map $\imath^* : H_{d-1}(\Delta) \to H_{d-1}(\Delta,\cost_\Delta(F))$
which appears in the long exact sequence of the pair $(\Delta,\cost_\Delta(F))$.
By using this fact,
we prove the following statement which proves the equivalence of (ii), (iii) and (iv) in Theorem \ref{main1}.

\begin{lemma}
\label{2.2}
Let $\Delta$ be a simplicial complex on $[n]$ of dimension $\geq 1$.
The following conditions are equivalent.
\begin{itemize}
\item[(i)] $\imath^* : H_{\dim \Delta}(\Delta) \to H_{\dim \Delta}(\Delta,\cost_\Delta(F))$ is non-zero for any face $F \in \Delta$.
\item[(ii)] There is an injective $S$-homomorphism $\phi : K[\Delta] \to \cano_{K[\Delta]}$ of degree $0$.
\item[(iii)] $\dim_K H_{\dim \Delta}(\Delta)> \dim_K H_{\dim \Delta}(\cost_\Delta(F))$
for any face (equivalently, facet) $F \in \Delta$.
\item[(iv)]
For any point $p \in |\Delta|$,
the natural map $\imath^* : H_{\dim \Delta}(|\Delta|) \to H_{\dim \Delta}(|\Delta|,|\Delta|-p)$
induced by the inclusion $\imath: |\Delta|-p \to |\Delta|$ is non-zero.
\end{itemize}
\end{lemma}

\begin{proof}
Since both $K[\Delta]$ and $\omega_{K[\Delta]}$ are squarefree $S$-modules
(c.f.\ \cite{Ya}),
a degree preserving injection $\phi : K[\Delta] \to \omega_{K[\Delta]}$ exists if and only if
there is an element $f \in \cano_{K[\Delta]}$ of degree $0$
such that $x_F f \ne 0$ for any $F \in \Delta$.
Since $K$ is infinite,
the latter condition is equivalent to saying that,
for any face (equivalently, facet) $F \in \Delta$, the multiplication map
\begin{align*}
\times x_F: (\cano_{K[\Delta]})_0 \to (\cano_{K[\Delta]})_{\ee_F}
\end{align*}
is non-zero.
Then Lemma \ref{2.1} proves the equivalence of (i) and (ii).
Also, the long exact sequence of the pair $(\Delta,\cost_\Delta(F))$
$$
0
\longrightarrow
H_{\dim \Delta}(\cost_\Delta(F))
\longrightarrow
H_{\dim \Delta}(\Delta)
\stackrel{\imath^*}{\longrightarrow}
H_{\dim \Delta}(\Delta,\cost_\Delta(F))
\longrightarrow
\cdots
$$
guarantees the equivalence of (i) and (iii).

The equivalence of (i) and (iv) is standard in topology.
Observe that, for any point $p \in |\Delta|$,
if $F \in \Delta$ is the smallest face (w.r.t.\ inclusion) such that 
$p$ lies in the convex hull of $\{\ee_i:i \in F\}$,
then $p$ line in its relative interior (note that the relative interior of a point is the point itself)
and $ H_{\dim \Delta} (|\Delta|,|\Delta|-p) \cong H_{\dim \Delta} (|\Delta|,|\cost_\Delta(F)|)$
(see \cite[Lemma 1.3]{Gr} and \cite[Lemma 3.3]{Mu}).
Consider the commutative diagram
\begin{align*}
H_{\dim \Delta}(|\Delta|) & \stackrel{\imath^*}{\longrightarrow} H_{\dim \Delta} (|\Delta|,|\cost_\Delta(F)|)\\
\downarrow \varphi \hspace{15pt}& \hspace{60pt}\downarrow \psi\\
H_{\dim \Delta}(|\Delta|) & \stackrel{\imath^*}{\longrightarrow} H_{\dim \Delta} (|\Delta|,|\Delta|-p)
\end{align*}
induced by the inclusion.
Then $\varphi$ is identity and $\psi$ is an isomorphism.
This proves the equivalence of (i) and (iv).
\end{proof}

To prove Theorem \ref{main1},
we need one more technical lemma.
Let $\Delta$ be a simplicial complex.
For a face $F \in \Delta$,
the simplicial complex
$$\lk_\Delta(F)=\{ G \in \Delta: F \cup G \in \Delta,\ F \cap G= \emptyset\}$$
is called the {\em link} of $F$ in $\Delta$.
We write $\lk_\Delta(v)=\lk_\Delta(\{v\})$ for simplicity.
A simplicial complex is said to be {\em pure} if all its facets have the same cardinality.

\begin{lemma}
\label{2.3}
Let $\Delta$ be a pure simplicial complex of dimension $\geq 2$ and $v$ a vertex of $\Delta$.
If $\Delta$ satisfies the condition (i) in Lemma \ref{2.2}
then so does $\lk_\Delta(v)$.
\end{lemma}

\begin{proof}
Let $d= \dim \Delta +1$ and let $F$ be a face of $\lk_\Delta(v)$.
By the purity of $\Delta$, $\dim \lk_\Delta(v)=d-2$.
Thus, what we must prove is that
the map
\begin{align}
\label{2-4}
\imath^* :H_{d-2}(\lk_\Delta(v)) \to H_{d-2}(\lk_\Delta(v), \cost_{\lk_\Delta(v)}(F))
\end{align}
is non-zero.
By the natural isomorphisms
$H_{d-2}(\lk_\Delta(v)) \cong H_{d-1}(\Delta,\cost_\Delta(v))$
and
$H_{d-2}(\lk_\Delta(v),\cost_{\lk_\Delta(v)}(F)) \cong H_{d-2}(\Delta,\cost_{\Delta}(F \cup \{v\})),$
the map \eqref{2-4} is identified with the map
\begin{align*}
\imath^* : H_{d-1}(\Delta,\cost_\Delta(v))
{\longrightarrow} H_{d-1}(\Delta,\cost_\Delta(F \cup \{v\}))
\end{align*}
in the long exact sequence of the triple
$(\Delta,\cost_\Delta(F \cup \{v\}),\cost_\Delta(v))$.
(Indeed, the natural isomorphisms send an element corresponding to $G \in \lk_\Delta(v)$ to that of $G \cup \{v\}$ in the homology, so they commute with $\imath^*$.)
This map  is non-zero since its composition with the map
$\imath^*:H_{d-1}(\Delta) \to H_{d-1}(\Delta,\cost_\Delta(v))$ is the map
$\imath^*:H_{d-1}(\Delta) \to H_{d-1}(\Delta,\cost_\Delta(F \cup \{v\}))$
which is non-zero by the assumption.
\end{proof}

We also recall Reisner's criterion of CM simplicial complexes \cite{Re}.

\begin{lemma}[Reisner's criterion]
A $(d-1)$-dimensional simplicial complex $\Delta$ is CM
if and only if, for any face $F \in \Delta$ (including the empty face $\emptyset$),
$\widetilde H_{k}(\lk_\Delta(F))=0$ for all $k \ne d-1-\#F$.
\end{lemma}

We now prove Theorem \ref{main1}.
For subsets $F_1,F_2,\dots,F_r$ of $[n]$,
let
$$
\langle F_1,F_2,\dots,F_r \rangle=\big\{ G \subset [n]:
G \subset F_k \mbox{ for some } k \in \{1,2,\dots,r\}\big\}
$$
be the simplicial complex generated by $F_1,F_2,\dots,F_r$.
Also, for $F \subset [n]$,
let $$\partial F= \langle F \setminus \{i\}: i \in F\rangle.$$

\begin{proof}[Proof of Theorem \ref{main1}]
The equivalence (ii) $\Leftrightarrow$ (iii) $\Leftrightarrow$ (iv)
follows from Lemma \ref{2.2}.
We prove (i) $\Leftrightarrow$ (iii).

Observe that $\widetilde H_i(\Delta)=0$ for all $i \ne d-1$ since $\Delta$ is CM.
Let $F$ be a facet of $\Delta$.
Then
we have the Mayer--Vietoris long exact sequence
\begin{eqnarray}
\label{2-3}
\begin{array}{lll}
&0\longrightarrow \widetilde H_{d-1}(\langle F \rangle) \bigoplus \widetilde H_{d-1}(\cost_\Delta(F))
\stackrel{\psi}{\longrightarrow} \widetilde H_{d-1}(\Delta)
\stackrel{\varphi}{\longrightarrow} \widetilde H_{d-2}(\partial F)\\
&\longrightarrow \widetilde H_{d-2}(\langle F \rangle) \bigoplus \widetilde H_{d-2}(\cost_\Delta(F))
\longrightarrow 0 
\longrightarrow \cdots.
\end{array}
\end{eqnarray}
Note that $\widetilde H_k(\langle F \rangle)=0$ for all $k$.

We first prove (i) $\Rightarrow$ (iii).
Suppose that $\cost_\Delta(F)$ is CM of dimension $d-1$.
Then $\widetilde H_{d-2}(\cost_\Delta(F))=0$.
Since $\widetilde H_{d-2}(\partial F) \cong K$,
the exact sequence \eqref{2-3} implies $\dim_K H_{d-1}(\cost_\Delta(F)) < \dim_K H_{d-1}(\Delta)$.

We prove (iii) $\Rightarrow$ (i) by induction on $d$.
When $d=2$, the implication easily follows from the fact that a $1$-dimensional simplicial complex is CM if and only if it is connected.
Suppose $d >2$.
We prove that $\cost_\Delta(F)$ is CM of dimension $d-1$.
Since $\Delta$ has the non-zero top homology by the assumption,
$\Delta$ is not a simplex and $\dim (\cost_\Delta(F))=d-1$.
Since $\dim_K \widetilde H_{d-1}(\cost_\Delta(F)) < \dim_K \widetilde H_{d-1}(\Delta)$,
$\widetilde H_{d-1}(\langle F \rangle)=0$ and $\widetilde H_{d-2}(\partial F) \cong K$,
the map $\varphi$ in \eqref{2-3} is surjective.
Hence $\widetilde H_{d-2}(\cost_\Delta(F))=0$.
The exact sequence \eqref{2-3} also shows $\widetilde H_k(\cost_\Delta(F))=0$ for $k < d-2$ since $\widetilde H_k(\Delta)=\widetilde H_k(\partial F) =0$ for $k \leq d-2$.
Then Reisner's criterion says that, to prove that $\cost_\Delta(F)$ is CM,
it is enough to prove that $\lk_{\cost_\Delta(F)}(v)$ is CM for any vertex $v$ of $\cost_\Delta(F)$.
A routine computation implies
\begin{align}
\label{66}
\lk_{\cost_\Delta(F)}(v)=
\begin{cases}
\lk_\Delta(v), & \mbox{ if } v \not \in F,\\
\cost_{\lk_\Delta(v)}(F\setminus \{v\}), & \mbox{ if } v \in F.
\end{cases}
\end{align}
By Reisner's criterion, $\lk_\Delta(v)$ is CM.
Also, $\lk_\Delta(v)$ satisfies condition (iii) by Lemmas \ref{2.2} and \ref{2.3}.
Then $\lk_\Delta(v)$ is uniformly CM by the induction hypothesis
and \eqref{66} proves that $\lk_{\cost_\Delta(F)}(v)$ is CM, as desired.
\end{proof}

\begin{remark}
A link of a uniformly CM simplicial complex is uniformly CM by \eqref{66}.
\end{remark}

\begin{remark}
There are non-CM simplicial complexes satisfying the equivalent conditions (ii)--(iv) of Theorem \ref{main1}.
For example,
Buchsbaum* complexes introduced by Athanasiadis and Welker \cite{AW}
satisfy the conditions. See \cite[Proposition 2.3]{AW}.
\end{remark}

In the rest of this section,
we discuss a few easy properties of uniformly CM simplicial complexes.
First, we consider an enumerative property.
For a $(d-1)$-dimensional simplicial complex $\Delta$,
its {\em $f$-vector} is the sequence
$$f(\Delta)=(f_{-1}(\Delta),f_0(\Delta),\dots,f_{d-1}(\Delta))$$
defined by $f_i(\Delta)=\#\{F \in \Delta: \# F=i+1\}$,
where $f_{-1}(\Delta)=1$,
and its {\em $h$-vector} $h(\Delta)=(h_0(\Delta),h_1(\Delta),\dots,h_d(\Delta))$
is defined by the relation
$$
\sum_{k=0}^d f_{k-1}(\Delta) (t-1)^{d-k}
= \sum_{k=0}^d h_k(\Delta) t^{d-k}.
$$
It is well-known that the $h$-vector of $\Delta$ and the Hilbert series of $K[\Delta]$
are related by
$$
\sum_{k=0}^\infty (\dim_K K[\Delta]_k) t^k = \frac 1 {(1-t)^d} \big(h_0(\Delta)+h_1(\Delta)t+ \cdots + h_d(\Delta)t^d \big).
$$

Let $\Delta$ be a $(d-1)$-dimensional uniformly CM simplicial complex with the $h$-vector $h(\Delta)=(h_0,h_1,\dots,h_d)$.
By Theorem \ref{main1},
there exists an injection $\phi: K[\Delta] \to \omega_{K[\Delta]}$.
Then the module $C=\coker \phi$ is CM of Krull dimension $d-1$ \cite[Lemma 3.1]{GTT} unless $C=0$.
Also, since the Hilbert series of $\omega_{K[\Delta]}$ is given by
$$
\frac 1 {(1-t)^d} \big(h_d(\Delta)+h_{d-1}(\Delta)t+ \cdots + h_0(\Delta)t^d \big),
$$
the short exact sequence 
$0 
\to
K[\Delta]
\to
\omega_{K[\Delta]}
\to
C 
\to
0$
says that the Hilbert series of $C$ is given by
\begin{align}
\label{2-5}
& \sum_{k=0}^\infty (\dim_K C_k) t^k = \frac 1 {(1-t)^{d-1}} \left\{\sum_{i=0}^{d-1} \big((h_d + \cdots +h_{d-i})-(h_0+ \cdots + h_i)\big)t^i\right\}.
\end{align}
Since the $h$-vector of a CM $S$-module is non-negative,
these facts imply
the following statement,
which is an analogue of Stanley's inequality for Ehrhart $\delta$-vectors \cite[Proposition 4.1]{St2}.

\begin{proposition}
\label{2.6}
If $\Delta$ is a $(d-1)$-dimensional uniformly CM simplicial complex
with the $h$-vector $h(\Delta)=(h_0,h_1,\dots,h_d)$,
then
$$h_d +h_{d-1} + \cdots +h_{d-i} \geq h_0+h_1+ \cdots + h_i $$
for all $i=0,1,\dots,d$.
\end{proposition}

Next, we discuss 2-CM simplicial complexes.
A $d$-dimensional simplicial complex $\Delta$ is said to be {\em doubly Cohen--Macaulay} ($2$-CM for short)
if it is CM and, for any vertex $v$ of $\Delta$,
$\cost_\Delta(v)$ is CM of dimension $d$.
Note that it is known that a CM simplicial complex $\Delta$ is $2$-CM if and only if $K[\Delta]$ is level and has $a$-invariant zero \cite[p.\ 94]{St}.
It is not hard to see that 2-CM simplicial complexes are uniformly CM
(see e.g.\ the proof of \cite[III Proposition 3.3]{St}).
The following result gives a more precise relation between these properties.

\begin{proposition}
\label{3.4}
If a simplicial complex $\Delta$ is $2$-CM,
then $\cost_\Delta(F)$ is CM for any face $F \in \Delta$.
\end{proposition}

To prove the above proposition,
we need the following well-known fact (see, e.g., the proof of \cite[Theorem 5.1.13]{BH}).

\begin{lemma}
\label{3.1}
Let $\Delta$ and $\Gamma$ be $d$-dimensional CM simplicial complexes.
If $\Delta \cap \Gamma$ is CM of dimension $d-1$,
then $\Delta \cup \Gamma$ is CM.
\end{lemma}

For a simplicial complex $\Delta$ and a new vertex $v$ which is not in $\Delta$,
let
$$v *\Delta = \Delta \cup \big\{ \{v\} \cup F: F \in \Delta \big\}$$
be the {\em cone} over $\Delta$ with apex $v$.
Note that if $\Delta$ is CM then so is $v*\Delta$.

\begin{proof}[Proof of Proposition \ref{3.4}]
We use induction on $d = \dim \Delta +1$.
The statement is obvious when $d=1$.
Suppose $d >1$.
Let $F \in \Delta$ be a face with $\#F>1$.
We prove that $\cost_\Delta(F)$ is CM.

Let $v \in F$.
A routine computation says
$$\cost_\Delta(F)=\cost_\Delta(v) \cup \big( v * \cost_{\lk_\Delta(v)}(F \setminus \{v\}) \big).$$
Since a link of a $2$-CM simplicial complex is 2-CM \cite[Proposition 9.7]{Wa},
by the induction hypothesis,
$\cost_{\lk_\Delta(v)}(F \setminus \{v\})$ is CM of dimension $d-2$.
Then, since $\cost_\Delta(v)$ is CM of dimension $d-1$
and since
$$\cost_\Delta(v) \cap (v * \cost_{\lk_\Delta(v)}(F \setminus \{v\}))
= \cost_{\lk_\Delta(v)}(F \setminus \{v\}),$$
it follows from Lemma \ref{3.1} that
$\cost_\Delta(F)$ is CM as desired.
\end{proof}

\begin{remark}
The proof of Proposition 2.8 can be simplified by using stellar subdivisions (see \cite[Definition 2.22]{Ko}).
Indeed, if $\Delta$ is $2$-CM and $F$ is its face, then the stellar subdivision $\Delta'$ of $\Delta$ with respect to $F$ is $2$-CM since $2$-CM property is a topological property.
Also, we have $\cost_{\Delta}(F)=\cost_{\Delta'}(v_F)$, where $v_F$ is the new vertex that is added when we take a stellar subdivision, so $\cost_{\Delta}(F)$ is CM.
\end{remark}

\section{Almost Gorenstein* simplicial complexes}

In this section,
we study almost Gorenstein* simplicial complexes.

\subsection{The almost Gorenstein* property}

Recall that a simplicial complex $\Delta$ is almost Gorenstein*
if $K[\Delta]$ is almost Gorenstein and $\widetilde H_{\dim \Delta}(\Delta)$ is non-zero.
In other words, $\Delta$ is almost Gorenstein* if it is CM and  there is an injection $\phi:K[\Delta] \to \omega_{K[\Delta]}$ of degree $0$ such that the multiplicity of $\coker \phi$ is equal to the number of its minimal generators.

Let $\Delta$ be a $(d-1)$-dimensional CM simplicial complex on $[n]$.
We write $\type(\Delta)$ for the {\em type} of $K[\Delta]$,
that is,
$$\type(\Delta)= \dim_K \Tor_{n-d}^S(K[\Delta],K).$$
Also, we define
$$\delta(\Delta)=\sum_{i=0}^{d-1} \{ (h_d+ \cdots+ h_{d-i})-(h_0+\cdots+h_i)\}.$$
Suppose that $\Delta$ is uniformly CM.
Then we have a short exact sequence
\begin{align}
\label{4-1}
0 \longrightarrow K[\Delta] \stackrel \phi \longrightarrow \omega_{K[\Delta]} \longrightarrow C=\coker \phi \longrightarrow 0.
\end{align}
By \eqref{2-5},
$\delta(\Delta)$ is the multiplicity of $C$.
Since $\type(\Delta)$ is equal to the number of minimal generators of $\omega_{K[\Delta]}$, the number of minimal generators of $C$ is $\type(\Delta)-1$.
Then, since the number of minimal generators of a CM $S$-module is smaller than or equal to its multiplicity, we have the following property.

\begin{lemma}
\label{4.1}
Let $\Delta$ be a uniformly CM simplicial complex.
Then $\type(\Delta) -1 \leq \delta(\Delta)$.
Moreover, equality holds if and only if $\Delta$ is almost Gorenstein*.
\end{lemma}

\begin{corollary}
\label{4.2}
A simplicial complex $\Delta$ is almost Gorenstein*
if and only if it is uniformly CM and $\delta(\Delta)=\type(\Delta)-1$.
\end{corollary}

We later use the following fact.

\begin{lemma}
\label{4.3}
Let $\Delta$ be a $(d-1)$-dimensional almost Gorenstein* simplicial complex on $[n]$
with the $h$-vector $h(\Delta)=(h_0,h_1,\dots,h_d)$.
Then
$$\dim_K \Tor_{n-d}^S (K[\Delta],K)_{n-k}= (h_d+ \cdots +h_{d-k})-(h_0+\cdots+h_k)$$
for all $k \geq 1$.
\end{lemma}

\begin{proof}
With the same notation as in \eqref{4-1},
we have
$$\Tor_{n-d}^S (K[\Delta],K)_{n-k} \cong \Tor_0^S (\omega_{K[\Delta]},K)_k \cong \Tor_0^S(C,K)_k$$
for all $k \geq 1$
(see \cite[I \S 12]{St} for the first isomorphism).
Let $\Theta$ be a linear system of parameters of $C$.
Since $C$ is a CM $S$-module whose multiplicity is equal to the number of its minimal generators, we have
$\Tor_0^S(C,K) \cong (C/\Theta C)$.
Then \eqref{2-5} implies
$$\dim_K \Tor_{n-d}^S (K[\Delta],K)_{n-k}=\dim_K (C/\Theta C)_k=(h_d+ \cdots +h_{d-k})-(h_0+\cdots+h_k)$$
for all $k \geq 1$, as desired.
\end{proof}

\subsection{Proof of the main result}
In this subsection, we prove Theorem \ref{main2}.
For a simplicial complex $\Delta$ on $[n]$
and a subset $W \subset [n]$,
let
$$\Delta|_W=\{F \in \Delta: F \subset W\}$$
be the {\em restriction} of $\Delta$ to $W$.
The Hochster's formula for $\Tor$-algebras \cite[Theorem 5.5.1]{BH}
gives a way to compute $\type(\Delta)$ from restrictions of $\Delta$.

\begin{lemma}[Hochster's formula]
\label{4.4}
Let $\Delta$ be a simplicial complex on $[n]$.
Then, for any non-negative integer $i \leq n$,
$\Tor_i^S(K[\Delta],K)=\bigoplus_{F \subset [n]} (\Tor_i^S(K[\Delta],K))_{\ee_F}$,
and for any $F \subset [n]$
one has
$$\Tor_i^S(K[\Delta],K)_{\ee_F} \cong \widetilde H_{\#F-i-1}(\Delta|_F).$$
\end{lemma}

Note that we consider that any simplicial complex contains the empty face $\emptyset$.
In particular, we consider $\widetilde H_{-1}(\Delta)=0$ if $\Delta \ne \{ \emptyset\}$ and
$\widetilde H_{-1}(\{\emptyset\})\cong K$ in the above lemma.

Next, we study basic properties of ridge sums.
Recall that a $(d-1)$-dimensional simplicial complex $\Delta$
is the ridge sum of $(d-1)$-dimensional simplicial complexes $\Gamma$ and $\Sigma$
if $\Delta=\Gamma \cup \Sigma$ and $\Gamma \cap \Sigma$ is a $(d-2)$-dimensional simplex.
See Figure 1.
\bigskip

\begin{center}
\includegraphics[width=120mm]{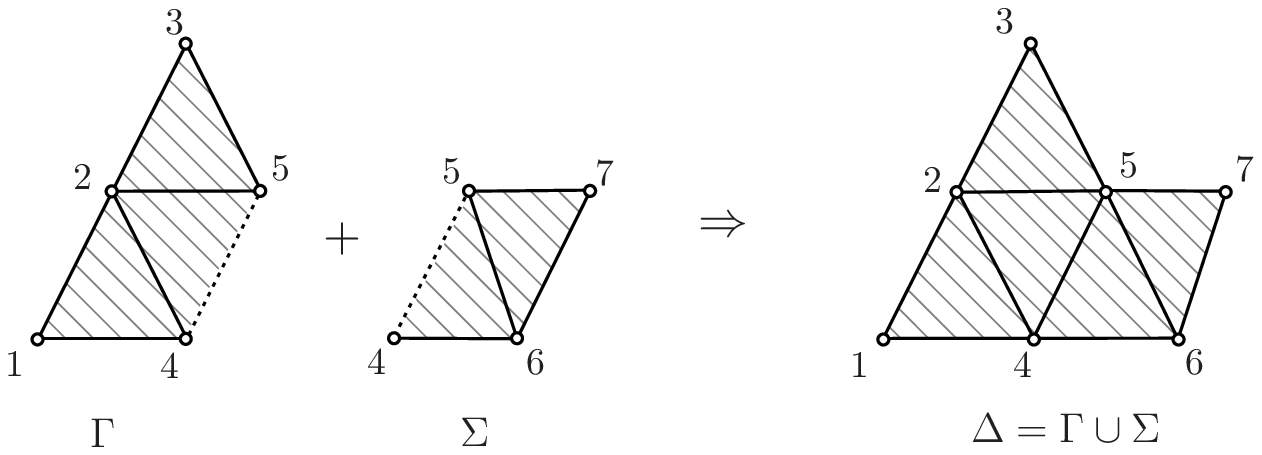}\\
Figure 1: The ridge sum of $\Gamma$ and $\Sigma$
\end{center}

\begin{lemma}
\label{4.5}
Let $\Delta$ be a simplicial complex on $[n]$ of dimension $d-1 \geq 1$.
Suppose that $\Delta$ is the ridge sum of $(d-1)$-dimensional simplicial complexes
$\Gamma$ and $\Sigma$.
Then
\begin{enumerate}
\item[(i)] $h(\Delta)=h(\Gamma)+h(\Sigma)+(-1,1,0,\dots,0)$.
\item[(ii)] $\mult(\Delta)=\mult(\Gamma)+\mult(\Sigma) +2$.
\item[(iii)] $\Delta$ is CM if and only if both $\Gamma$ and $\Sigma$ are CM.
\item[(iv)] $\Delta$ is uniformly CM if and only if both $\Gamma$ and $\Sigma$ are uniformly CM.
\item[(v)] Suppose that $\Delta$ is CM.
If neither $\Gamma$ nor $\Sigma$ is a simplex,
then $\type(\Delta)=\type(\Gamma)+\type(\Sigma)+1$.
\item[(vi)] $\Delta$ is almost Gorenstein*
if and only if both $\Gamma$ and $\Sigma$ are almost Gorenstein*.
\end{enumerate}
\end{lemma}

\begin{proof}
Let $\Gamma \cap \Sigma=\langle W \rangle$.
Note that $\#W=d-1$ and $K[\langle W \rangle ] \cong S/(x_i:i\not\in W)$.
Then the Mayer--Vietoris short exact sequence
\begin{align}
\label{3-1}
0 \longrightarrow K[\Delta] \longrightarrow K[\Gamma] \bigoplus K[\Sigma] \longrightarrow K[\langle W \rangle] \longrightarrow 0
\end{align}
says
\begin{align*}
\frac 1 {(1-t)^d} \left( \sum_{i=0}^d h_i(\Delta)t^i \right)
= \frac 1 {(1-t)^d} \left( \sum_{i=0}^d \big(h_i(\Gamma) + h_i(\Sigma)\big)t^i \right)
-\frac 1 {(1-t)^{d-1}}.
\end{align*}
The above equation implies (i) and (ii).

We prove (iii).
The `if' part is a special case of Lemma \ref{3.1}.
We prove the `only if' part.
Recall that a graded $K$-algebra $R=S/I$ of Krull dimension $d$ is CM if and only if $\Tor_i^S(R,K)=0$ for $i > n-d$.
Suppose that $\Delta$ is CM.
Since $K[\Delta]$ is CM of Krull dimension $d$ and $K[\langle W\rangle]$ is CM of Krull dimension $d-1$,
the short exact sequence \eqref{3-1} says
$\Tor_{i}^S(K[\Gamma],K)=
\Tor_{i}^S(K[\Sigma],K)=0$ for $i > n-d+1$.
What we must prove is
$\Tor_{n-d+1}^S(K[\Gamma],K)=
\Tor_{n-d+1}^S(K[\Sigma],K)=0$.

The short exact sequence \eqref{3-1}
induces the long exact sequence
\begin{align}
\nonumber
0 
&
\longrightarrow
\Tor_{n-d+1}^S(K[\Gamma],K)
\bigoplus
\Tor_{n-d+1}^S(K[\Sigma],K)
\longrightarrow
\Tor_{n-d+1}^S(K[\langle W \rangle],K)
\\
\label{3-2}
&\longrightarrow
\Tor_{n-d}^S(K[\Delta],K)
\longrightarrow
\Tor_{n-d}^S(K[\Gamma],K)
\bigoplus
\Tor_{n-d}^S(K[\Sigma],K)
\\
&
\nonumber
\stackrel{\psi}{\longrightarrow}
\Tor_{n-d}^S(K[\langle W \rangle],K)
\longrightarrow \cdots.
\end{align}
Since $K[\langle W \rangle] \cong S/(x_i: i \not \in W)$,
we have
$$\Tor_{n-d+1}^S(K[\langle W \rangle],K)=\Tor_{n-d+1}^S(K[\langle W \rangle],K)_{n-d+1}.$$
Then the long exact sequence \eqref{3-2} says
$$\Tor_{n-d+1}^S(K[\Gamma],K)=\Tor_{n-d+1}^S(K[\Gamma],K)_{n-d+1}$$
and
$$\Tor_{n-d+1}^S(K[\Sigma],K)=\Tor_{n-d+1}^S(K[\Sigma],K)_{n-d+1}.$$
On the other hand, for any graded ideal $I \subset S$,
one has $\Tor_{n-d+1}^S(S/I,K)_{n-d+1} \ne 0$ if and only if
$I$ contains $n-d+1$ linearly independent linear forms.
Since $\Gamma$ and $\Sigma$ have dimension $d-1$,
if we write $K[\Gamma]=S/I_{\Gamma}$ and $K[\Sigma]=S/I_\Sigma$ then $I_\Gamma$ and $I_\Sigma$ contain
at most $n-d$ linearly independent linear forms.
Thus
$\Tor_{n-d+1}^S(K[\Gamma],K)=
\Tor_{n-d+1}^S(K[\Sigma],K)=0$.

(iv) is a direct consequence of (iii)
since, for any facet $F \in \Delta$,
one has $\cost_\Delta(F)=\cost_\Gamma(F) \cup \cost_\Sigma(F)$
and $\cost_\Gamma(F) \cap \cost_\Sigma(F)=\langle W \rangle$,
where we consider that $\cost_\Gamma(F)=\Gamma$ when $F \not \in \Gamma$.

Next, we prove (v).
Suppose that $\Delta$ is CM.
Then $\Gamma$ and $\Sigma$ are CM by (iii).
We claim that the map $\psi$ in \eqref{3-2} is the zero map.
Since
$K[\langle W \rangle] \cong S/(x_i: i \not \in W)$,
$\Tor_{n-d}(K[\langle W \rangle],K)$ has non-zero elements only in degree $n-d$.
Thus
\begin{align}
\label{3-3}
\Tor_{n-d}^S(K[\langle W \rangle],K)
=\bigoplus_{\substack{F \subset [n] \\ \# F =n-d}}
\left( \Tor_{n-d}^S(K[\langle W \rangle],K)
\right)_{\ee_F}.
\end{align}
On the other hand, Hochster's formula says that,
for any $F \subset [n]$ with $\# F=n-d$,
$$
\Tor_{n-d}^S(K[\Gamma],K)_{\ee_F} \cong \widetilde H_{-1}(\Gamma |_F)
\mbox{ and }
\Tor_{n-d}^S(K[\Sigma],K)_{\ee_F} \cong \widetilde H_{-1}(\Sigma |_F).
$$
Since $\Gamma$ and $\Sigma$ are not simplexes,
$\Gamma|_F \ne \{\emptyset\} $ and $\Sigma|_F \ne \{\emptyset\}$
if $\#F \geq n-d$.
Thus we have
$$
\Tor_{n-d}^S(K[\Gamma],K)_{\ee_F}
=\Tor_{n-d}^S(K[\Sigma],K)_{\ee_F}=0$$
for all $F \subset [n]$ with $\# F =n-d$.
This fact and \eqref{3-3} say that $\psi$ is the zero map.
Then, by \eqref{3-2}, we have
{\small
\begin{align*}
\type (\Delta)=&
\dim_K \Tor_{n-d}^S (K[\Delta],K)\\
=&\dim_K \big(\Tor_{n-d}^S (K[\Gamma],K) \bigoplus \Tor_{n-d}^S(K[\Sigma],K) \big)+ \dim_K \Tor_{n-d+1}^S(K[\langle W \rangle],K)\\
=& \type(\Gamma)+\type(\Sigma) + 1,
\end{align*}
}
as desired.

It remains to prove (vi).
By (iv), we may assume that $\Delta$, $\Gamma$ and $\Sigma$ are uniformly CM.
Note that this guarantees that $\Gamma$ and $\Sigma$ are not simplexes.
By (ii) and (v) ,
$$\mult(\Delta)=\mult(\Gamma)+\mult(\Sigma)+2$$
and
$$\type(\Delta)-1=\type(\Gamma)+\type(\Sigma).$$
Then the assertion follows from Lemma \ref{4.1}.
\end{proof}

\begin{lemma}
\label{ridgesum}
Let $\Delta$ be a $(d-1)$-dimensional CM simplicial complex on $[n]$.
If $\Tor_{n-d}^S(K[\Delta],K)_{n-d+1} \ne 0$
then there are $(d-1)$-dimensional simplicial complexes $\Gamma$ and $\Sigma$
such that $\Delta$ is the ridge sum of $\Gamma$ and $\Sigma$.
\end{lemma}

\begin{proof}
Since Hochster's formula says that 
$$\Tor_{n-d}^S(K[\Delta],K)_{n-d+1}\cong \bigoplus_{{F \subset[n],\ \#F=n-d+1}} \widetilde H_0(\Delta|_F),$$
there is a subset $W \subset [n]$ with $\#W=d-1$ such that
$\Delta|_{[n]\setminus W}$ has more than two connected components.
Let
$\Gamma'$ be a connected component of
$\Delta|_{[n]\setminus W}$
and let $\Sigma' =\{F \in \Delta|_{[n]\setminus W} : F \not \in \Gamma'\}$
be the complement of $\Gamma'$ in $\Delta|_{[n]\setminus W}$.
Let $V$ be the vertex set of $\Gamma'$ and $V^c$ the vertex set of $\Sigma'$.
We claim that $\Delta$ is the ridge sum of $\Gamma=\Delta|_{V \cup W}$
and $\Sigma=\Delta|_{V^c \cup W}$.

We first prove $\Delta=\Gamma \cup \Sigma$.
It is enough to prove that $\Delta \subset \Gamma \cup \Sigma$.
Let $F$ be a facet of $\Delta$.
Since $\Gamma'=\Delta|_{V}$ is a connected component of $\Delta|_{[n] \setminus W}$,
one has either $F \setminus W \subset V$ or $F \setminus W \subset V^c$,
which implies $F \in \Gamma \cup \Delta$.
Thus $\Delta \subset \Gamma \cup \Sigma$.
Next, it is clear that $\Gamma$ and $\Sigma$ have dimension $d-1$
since all facets of $\Delta$ have dimension $d-1$.
It remains to prove that $\Gamma \cap \Sigma$ is generated by $W$.
To prove this, what we must prove is that $W \in \Delta$.
Let $F$ and $G$ be facets of $\Delta$ such that $F \in \Gamma$ and $G \in \Sigma$.
Note that such facets exist since $\Gamma$ and $\Sigma$ are not empty.
Since $\Delta$ is CM,
$\Delta$ is strongly connected \cite[Proposition 11.7]{Bj}.
Thus there is a sequence of facets $F=F_0,F_1,\dots,F_t=G$
such that $\#F_i\cap F_{i+1}=d-1$ for $i=0,1,\dots,t-1$.
Then, since $\Gamma$ is a simplicial complex on $V \cup W$
and $\Sigma$ is a simplicial complex on $V^c \cup W$,
if we choose an integer $j$ such that $F_j \in \Gamma$ and $F_{j+1} \in \Sigma$,
then we have $F_j \cap F_{j+1} = W$.
Hence $W \in \Delta$.
\end{proof}

Now we prove Theorem \ref{main2}.

\begin{proof}[Proof of Theorem \ref{main2}]
Suppose $\widetilde H_{d-1}(\Delta) \not \cong K$.
Then $\dim_K \widetilde H_{d-1}(\Delta) >1$.
Since $h_d(\Delta)=\sum_{k=0}^d (-1)^{d-k} f_{k-1}(\Delta)=\dim_K \widetilde H_{d-1}(\Delta)$ by Reisner's criterion,
we have
$$\dim_K \Tor_{n-d}(K[\Delta],K)_{n-d+1}= h_d-1>0$$
by Lemma \ref{4.3}.
Then the desired statement follows from Lemmas \ref{4.5}(vi) and \ref{ridgesum}.
\end{proof}

\begin{remark}
A simplicial complex $\Delta$ is said to be {\em Gorenstein*} if $K[\Delta]$ is Gorenstein
and $\widetilde H_{\dim \Delta}(\Delta) \cong K$. Lemma \ref{4.5}(vi) says that a ridge sum of Gorenstein* simplicial complexes are almost Gorenstein*.
In particular, this gives an infinite family of Stanley--Reisner rings which are almost Gorenstein but not Gorenstein.
\end{remark}

\subsection{Almost Gorenstein* simplicial complexes of dimension $\leq 2$}

By Theorem \ref{main2},
to study almost Gorenstein* simplicial complexes
it is crucial to study indecomposable ones.
A typical example of an indecomposable almost Gorenstein* simplicial complex
is a Gorenstein* simplicial complex.
For $1$-dimensional simplicial complexes, that is graphs,
it is not hard to see that any indecomposable almost Gorenstein*
simplicial complex is actually Gorenstein* (i.e., it is a cycle).
Indeed, such a simplicial complex $\Delta$
has the $h$-vector $(1,h_1,1)$, which implies that $\delta(\Delta)=0$ and $K[\Delta]$ is Gorenstein.
Thus any almost Gorenstein* graph is a ridge sum of cycles, that is,
obtained from a cycle by taking a ridge sum with a cycle repeatedly (see Figure 2).
This fact and \cite[Theorem 10.4]{GTT} give the following characterization of the almost Gorenstein Stanley--Reisner rings of Krull dimension $2$.

\begin{center}
\includegraphics[width=40mm]{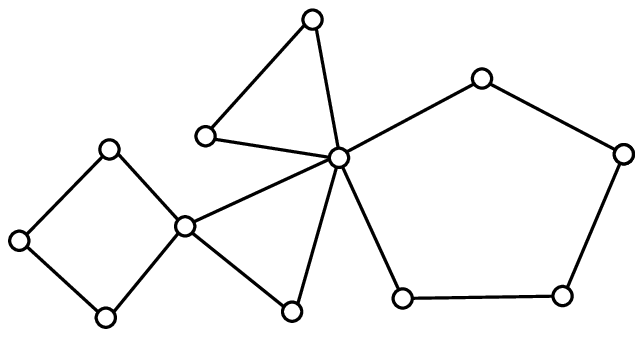}\\
Figure 2: A ridge sum of cycles
\end{center}

\begin{proposition}
\label{4.7}
Let $\Delta$ be a $1$-dimensional simplicial complex.
Then $K[\Delta]$ is almost Gorenstein if and only if either $\Delta$ is a tree or a ridge sum of cycles.
\end{proposition}

\begin{proof}
If $\widetilde H_1(\Delta)$ is non-zero,
then $K[\Delta]$ is almost Gorenstein if and only if $\Delta$ is a ridge sum of cycles.
On the other hand, $\Delta$ is a tree
if and only if $\Delta$ is CM and $h_2(\Delta)=f_1(\Delta)-f_0(\Delta)+1=0$.
Then the assertion follows from \cite[Theorem 10.4]{GTT}
which proved that if $K[\Gamma]$ is CM and $h_2(\Gamma)=0$,
then $K[\Gamma]$ is almost Gorenstein.
\end{proof}

In dimension $2$,
there is an indecomposable almost Gorenstein* simplicial complex
which is not Gorenstein* (see Example \ref{5.6}).
Thus an almost Gorenstein* simplicial complex is not always a ridge sum of Gorenstein* simplicial complexes.
On the other hand,
the next statement and Theorem \ref{main2} provide a combinatorial criterion of the $2$-dimensional almost Gorenstein* simplicial complexes.

\begin{proposition}
\label{4.8}
Let $\Delta$ be a $2$-dimensional simplicial complex.
Then $\Delta$ is indecomposable almost Gorenstein* if and only if $\Delta$ is uniformly CM and $\widetilde H_2(\Delta) \cong K$.
\end{proposition}

\begin{proof}
The only if part is obvious.
Suppose that $\Delta$ is uniformly CM and $\widetilde H_2(\Delta) \cong K$.
Since $h_3(\Delta)= \dim_K \widetilde H_2(\Delta)$,
the $h$-vector of $\Delta$ is the vector of the form $h(\Delta)=(1,h_1,h_2,1)$.
Let $\phi:K[\Delta] \to \omega_{K[\Delta]}$ be an injection and $C=\coker \phi$.
By \eqref{2-5},
the Hilbert series of $C$ is $ (h_2-h_1) t\over (1-t)^2$,
which implies that $C$ has $(h_2-h_1)$ generators of degree $1$.
Since $\delta(\Delta)=h_2-h_1$ is the multiplicity of $C$,
$K[\Delta]$ is almost Gorenstein.
\end{proof}

\subsection{Questions and examples}
We close this paper with a few questions and examples.
Let $\Delta$ be an almost Gorenstein* simplicial complex
with the $h$-vector $h(\Delta)=(h_0,h_1,\dots,h_d)$.
For $i=0,1,\dots,d-1$,
we define
$$ \eta_i(\Delta)= (h_d+ \cdots + h_{d-i}) - (h_0+\dots+h_i)$$
and
$$\eta_\Delta(t)= \eta_0 (\Delta)+ \eta_1 (\Delta)t + \cdots + \eta_{d-1}(\Delta)t^{d-1}.$$
Thus, as we saw in \eqref{2-5},
$\eta_\Delta(t)$ is the $h$-polynomial of the cokernel of the injection
$\phi: K[\Delta] \to \omega_{K[\Delta]}$
and $\delta(\Delta)=\eta_\Delta(1)$.
By the definition,
$\eta_\Delta(t)$ is symmetric,
that is, $\eta_k(\Delta)=\eta_{d-k-1}(\Delta)$ for all $k$.
Since the cokernel of $\phi$ looks to be of fundamental in the study of the almost Gorenstein* property,
to find various types of almost Gorenstein* simplicial complexes,
we ask the following problem.

\begin{problem}
\label{5.1}
Fix an integer $d$.
Characterize all symmetric polynomials 
$f=\eta_0 + \eta_1 t + \cdots + \eta_{d-1}t^{d-1} \in \ZZ_{\geq 0}[t]$
such that there is a $(d-1)$-dimensional almost Gorenstein* simplicial complex $\Delta$
with $\eta_\Delta(t)=f$.
\end{problem}

We give two examples.

\begin{example}
Lemma \ref{4.5} and \eqref{2-5} say that, if $\Delta$ is a ridge sum of $\eta$ Gorenstein* simplicial complexes of dimension $d-1$
then $\eta_\Delta(t)=\eta-1+(\eta-1) t^{d-1}$.
\end{example}

\begin{example}
\label{5.6}
We give a $2$-dimensional almost Gorenstein* simplicial complex with $\eta_1 \ne 0$.
Let
$$
\Sigma= \langle 134,125,235,345,126,136,236,146,156,456 \rangle.
$$
(See Figure 3.)
Then $\Sigma$ is uniformly CM and $\widetilde H_2(\Sigma) \cong K$.
Thus $\Sigma$ is almost Gorenstein* by Proposition \ref{4.8}.
Since $h(\Sigma)=(1,3,5,1)$, we have $\eta_\Sigma(t)=2t$.
\begin{center}
\includegraphics[width=40mm]{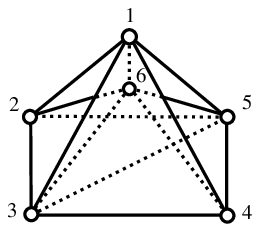}\\
Figure 3: The simplicial complex $\Sigma$
\end{center}
\end{example}

Below, we give more specified questions relating to Problem \ref{5.1}.
In the above example, we have $\delta(\Sigma)=2$.
Since $\delta(\Delta)$ is even if $\Delta$ is a ridge sum of Gorenstein* simplicial complexes, we suggest the following question.

\begin{question}
\label{q1}
Is $\delta(\Delta)$ even for any almost Gorenstein* simplicial complex $\Delta$?
\end{question}

The next question will be the first step to study Problem \ref{5.1}.

\begin{question}
\label{q2}
For all non-negative integers $i$ and $d$ with $i < \frac d 2$,
is there a $(d-1)$-dimensional almost Gorenstein* simplicial complex $\Delta$
with $\eta_\Delta(t)=t^i+t^{d-1-i}$?
\end{question}

Note that Question \ref{q1} is obvious in odd dimensions since $\eta_\Delta(t)$ is symmetric.
But the question is open even for $2$-dimensional simplicial complexes.
We think that Question \ref{q2} for $i=1$ would not be hard, but for $i \geq 2$ the question seems not to be easy.
Indeed, we do not have an example of an almost Gorenstein* simplicial complex of dimension $\geq 4$ with $\eta_2 \ne 0$.

Finally we give a few examples which show that the almost Gorenstein* property is combinatorially not a good property.

\begin{example}
The almost Gorenstein* property is not a topological property.
For example,
$$\Delta= \langle 123,124,134,234,345,346,356,456 \rangle$$
is almost Gorenstein* since it is a ridge sum of two boundaries of a simplex.
On the other hand,
$$\Delta'= \langle 123,124,137,147,237,247,357,457,367,467,356,456 \rangle$$
is a subdivision of $\Delta$ (subdivide the edge $\{3,4\}$ by a new vertex $7$)
but is not almost Gorenstein*
since it is not a ridge sum although $\dim_K \widetilde H_2(\Delta')=2 >1$.
\end{example}

\begin{example}
The link of an almost Gorenstein* simplicial complex may not be almost Gorenstein*.
Let
$$\Delta=
\langle 134,125,235,345,127,267,137,367,236,147,467,157,567,456 \rangle.$$
Then $\Delta$ is a subdivision of $\Sigma$ in Example \ref{5.6} (subdivide the edge $\{1,6\}$ by a new vertex $7$),
and is almost Gorenstein* by Proposition \ref{4.8}.
However,
$$\lk_\Delta(7)=\langle 12,26,13,36,14,46,15,56 \rangle$$
is not almost Gorenstein* by Proposition \ref{4.7}.
\end{example}

\end{document}